\documentclass[12pt]{amsart}
\usepackage[utf8]{inputenc}
\usepackage{amsmath,amsthm,amssymb}
\usepackage{xcolor}

\usepackage[british]{babel}

\usepackage[shortlabels]{enumitem}
\usepackage{hyperref}
\usepackage{mathtools}
\usepackage{graphicx}

\newtheorem{theorem}{Theorem}[section]

\newtheorem{lemma}[theorem]{Lemma}

\newtheorem{claim}[theorem]{Claim}
\mathtoolsset{showonlyrefs}

\newcounter{stepctr}
\newif\ifinsteppage

\newenvironment{steps}{%
	\setcounter{stepctr}{0}%
	\insteppagefalse
}{%
	\par\bigskip
}

\newcommand{\step}[1]{%
	\ifinsteppage
	\par\bigskip
	\fi
	\insteppagetrue
	\refstepcounter{stepctr}%
	\noindent\textbf{Step \thestepctr:}\if\relax\detokenize{#1}\relax\else\space#1.\fi\par
}

\newcommand{\demm}{\de^{-\e}}

\numberwithin{equation}{section}

\newcommand{\R}{\mathbb{R}}

\newcommand{\dm}{\de^{-O(\eta)}}

\newcommand{\ka}{\kappa}

\newcommand{\de}{\delta}
\newcommand{\e}{\epsilon}

\newcommand{\sg}{\sigma}

\newcommand{\dmm}{\de^{-\eta}}

\newcommand{\cn}[1]{N_\de({#1})}
\newtheorem*{ack*}{Acknowledgment}

\author{William O'Regan}
\begin{document}
	\title{A note on the asymmetric discretised ring theorem}

	\thanks{WOR is supported in part by an NSERC Alliance grant administered by Pablo Shmerkin and Joshua Zahl}
	\begin{abstract}
We prove an asymmetric version of the discretised ring theorem.
	\end{abstract}
	\maketitle
    \section{Introduction}
    The aim of this note is to prove the following.
    \begin{theorem}\label{thm.main}
        Let $0 < \ka < \sg < 1$. There exists $\e > 0$ so that the following holds for all $\de > 0$ small enough.
        
        Let $A,B,C \subset [1,2]$ be $\de$-separated. Suppose that $|A| = \de^{-\sg}$ and suppose that $B,C$ are $(\de,\ka,\demm)$-sets. Then
        \begin{equation}\label{eq.result}
            \cn{A+B} + \cn{AC} > \demm |A|.
        \end{equation}
    \end{theorem}
    Here $\cn{\ \cdot \ }$ denotes the $\de$-covering number.
    To the best of our knowledge, only when $2\kappa > \sigma$ is the theorem above covered in the literature. This is done by an application of the $ABC$ sum-product theorem \cite[Theorem 1.7]{orpshabc}.  We leave this deduction as an exercise to the interested reader. 

    The case when $A=B=C$ was proved by Bourgain in \cite{bou03} when $\kappa = \sigma,$ motivated by the connections between the discretised ring theorem and other problems in geometric combinatorics found by Katz and Tao in \cite{kat}, and then for $\kappa < \sigma$ by Bourgain and Gamburd in \cite{bougam}. Since then a variety of alternative proofs have emerged. See for example \cite{gut}, \cite{o2025simple}, \cite{ORegan2026}.  

    In finite fields the analogue of Theorem \ref{thm.main} was obtained by Bourgain \cite[Theorem 1.1]{Bourgain2005SumProduct}. This used an exponential sum estimate obtained in \cite{Bourgain2005DiffieHellman} which built upon the symmetric sum-product theorem in prime fields obtained in \cite{BourgainKatzTao2004}. Our proof of Theorem \ref{thm.main} adapts his strategy to the discretised setting, using a pointwise Fourier decay theorem for multiplicative convolutions \cite[Theorem 6]{bou}, \cite[Theorem 1.5]{OrponenDeSaxceShmerkin2023}, in place of the exponential sum estimate.

    \subsection*{Acknowledgements} We thank Tuomas Orponen for bringing this to our attention. This material is based upon work supported by the Swedish Research Council under grant no. 2021-06594 whilst the author was in residence at Institut Mittag-Leffler in Djursholm, Sweden during the autumn of 2026. Thanks are given to the organisers, administrators and the participants. Thanks are extended to Pablo Shmerkin and Joshua Zahl for support.
\section{Preliminary definitions and results}
\subsection*{Notation and non-concentration}
For a finite set $E$, the notation $|E|$ denotes its cardinality or Lebesgue measure, and it ought to be clear from context. For $r>0$, let
\[
    E_r=\{x\in\R:d(x,E)<r\},\qquad
    d(x,E)=\inf_{y\in E}|x-y|.
\]

For a bounded set $E\subset\R$, let $N_r(E)$ be the least number of intervals of length $r$ required to cover $E$.

For a positive integer $k$, the notation
\[
    E^{(k)}=\left\{\prod_{j=1}^k x_j:x_1,\ldots,x_k\in E\right\}
\]
denotes a product set, whereas $E^k=E\times\cdots\times E$ denotes a Cartesian product. 

We write $X\lesssim Y$ if $X\leq CY$ for an implicit constant $C>0$, and $X\sim Y$ if both $X\lesssim Y$ and $Y\lesssim X$. Implicit constants are independent of $\de$; any permitted dependence on other parameters is specified where needed. The notation $O(t)$ denotes a quantity bounded in absolute value by a constant multiple of $t$. In expressions such as $\de^{O(\eta)}$ and $\de^{-O(\eta)}$, the constants in the exponents are positive and may change from line to line.

Let $0<\ka<1$ and $C\geq1$. We say $E \subset \R$ is a $(\de,\ka,C)$-set if it obeys the non-concentration condition
\[
    \cn{E\cap B(x,r)}\leq Cr^\ka\cn E
    \qquad\text{for }x \in \R \text{ and } \de\leq r \leq1.
\]
In particular, $\cn E\geq C^{-1}\de^{-\ka}$. 
A Borel probability measure $\mu$ is a $(\de,\ka,C)$-measure if
\[
    \mu(B(x,r)) \leq Cr^\ka
    \qquad\text{for } x \in \R \text{ and } \de\leq r\leq1.
\]
For a finite non-empty set $E$, its uniform counting measure is $$|E|^{-1}\sum_{x\in E}\delta_x, $$where $\delta_x$ is the unit point mass at $x$. This is a $(\de,\ka,C)$-measure whenever $E$ is a $\de$-separated $(\de,\ka,C)$-set. By contrast, the uniform probability measure on a measurable set $U$ with $0<|U| <\infty$ is
\[
    \frac{\mathbf 1_U(x)}{|U|}\,dx.
\]
In particular, the uniform measure on $E_\de$ is the normalised Lebesgue measure on that neighbourhood.

\subsection*{Uniform subsets}
We recall the definition of uniform sets. Fix a bounded  interval $I$ and an integer $T\geq1$. Subdivide each interval into $2^T$ equal children at every step, stopping at the first level $q$ whose intervals have length at most $\de$. A $\de$-separated set $E\subset I$ is \emph{uniform} for these subdivisions if, for each $1\leq j\leq q$, every interval at level $j-1$ meeting $E$ has the same number of children meeting $E$.

The following is standard and may be found in a number of references. See, for example,  \cite[Lemma 3.6]{Shmerkin2019}.
\begin{lemma}\label{lem.uniform-subset}
For every $\e>0$, one may choose $T$ so that, for all sufficiently small $\de>0$, every non-empty $\de$-separated set $E$ contains a uniform subset $E'$ with $|E'|\geq\de^\e|E|$.
\end{lemma}

\subsection*{Fourier transforms and energy}
We use the convention
\[
    e(t)=e^{2\pi it},\qquad
    \widehat\mu(\xi)=\int_\R e(-\xi x)\,d\mu(x)
\]
for the Fourier transform of a finite Borel measure. The multiplicative convolution $\mu\boxtimes\nu$ is the image of $\mu\times\nu$ under $(x,y)\mapsto xy$; equivalently,
\[
    \int f\,d(\mu\boxtimes\nu)=\iint f(xy)\,d\mu(x)\,d\nu(y)
\]
for bounded Borel functions $f$. We write $\mu^{\boxtimes n}$ for the $n$-fold multiplicative convolution, so it is the image of $\mu^n$ under $(x_1,\ldots,x_n)\mapsto\prod_{j=1}^n x_j$.

For $0<s<1$, the $s$-energy of a finite positive Borel measure is
\[
    I_s(\mu)=\iint |x-y|^{-s}\,d\mu(x)\,d\mu(y).
\]
To define its regularised energy, fix a non-negative smooth function $\psi$, supported on $[-1,1]$, with $\int\psi=1$, and put $\psi_r(x)=r^{-1}\psi(x/r)$. The additive convolution with this function has density
\[
    (\mu*\psi_r)(x)=\int\psi_r(x-y)\,d\mu(y),
\]
and we set $I_s^r(\mu)=I_s((\mu*\psi_r)\,dx)$. Different fixed choices of such a smoothing function give comparable regularised energies. The standard Fourier representation of energy is
\[
    I_s(\mu)=c_s\int_\R|\widehat\mu(\xi)|^2|\xi|^{s-1}\,d\xi,
\]
where $c_s>0$ depends only on $s$. Since $|\widehat\psi(r\xi)|\leq1$, it follows that $I_s^r(\mu)\leq I_s(\mu)$ for every $r>0$.

\subsection*{Fourier decay}
We recall a consequence of our key external result \cite[Theorem 1.5]{OrponenDeSaxceShmerkin2023}, which is a quantitative version of \cite[Theorem 6]{bou}. It is not important for us that the result quantifies the number of multiplicative convolutions required, but we nevertheless choose to state the stronger of the two statements.

\begin{theorem}\label{thm.multiplicative-fourier-decay}
    Let $0<\kappa<1$ and let $n$ be a positive integer
    satisfying $n\kappa>1$. There exist $\epsilon,\tau>0$,
    depending only on $\kappa$ and $n$, such that the following
    holds for all $\delta>0$ sufficiently small.

    Let $\mu$ a be Borel probability measure
    supported on $[-1,1]$ satisfying
    \[
        I_\kappa^\delta(\mu)\leq\delta^{-\epsilon}.
    \]
    Then, for every $\xi\in\mathbb{R}$ with
    $\delta^{-1}\leq|\xi|\leq2\delta^{-1}$,
    \[
        \left|
        \int
        e^{-2\pi i\xi x_1\cdots x_n}
        \,d\mu(x_1)\cdots d\mu(x_n)
        \right|
        \leq|\xi|^{-\tau}.
    \]
\end{theorem}
We are able to obtain the following useful corollary which gives pointwise results over a larger range of frequencies with little additional work. Compare also with \cite[Corollary 1.10]{OrponenDeSaxceShmerkin2023}.

    \begin{theorem}\label{thm.frequencies}

    Let $0<\kappa<1$, and let $n$ be a positive integer satisfying
    $n\kappa>1$. There exist $\epsilon_0,\tau>0$,
    depending only on $\kappa$ and $n$, such that the following
    holds for every $0<\epsilon\leq\epsilon_0$, every fixed $R\geq 2$, and all $\delta>0$ sufficiently small.

    Let $C\subset[1,R]$ be a $(\de,\kappa,\demm)$-set, and let $\mu$ be the uniform probability measure on its $\de$-neighbourhood. Then, for every $\xi\in\mathbb{R}$ satisfying
    $\de^{-\sqrt{\e}} \le |\xi| < \infty$, we have
    \[
        \left|\widehat{\mu^{\boxtimes n}}(\xi)\right|
        =
        \left|
            \int
            e^{-2\pi i\xi x_1\cdots x_n}
            \,d\mu(x_1)\cdots d\mu(x_n)
        \right|
        \le |\xi|^{-\tau}.
    \]
    \begin{proof}
         The non-concentration condition and the thickening give
        \[
        \mu(B(x,r))\lesssim
        \begin{cases}
            \demm r^\kappa,& \de\leq r\leq 1,\\
            \demm\de^{\kappa-1}r,& 0<r<\de.
        \end{cases}
        \]
        Consequently, the full energy satisfies
        \[
            I_\kappa(\mu):=\iint |x-y|^{-\kappa}\,d\mu(x)\,d\mu(y)
            \lesssim \demm\log(1/\de).
        \]
        Let $\epsilon_1>0$ be the energy exponent in Theorem \ref{thm.multiplicative-fourier-decay} and choose $\epsilon_0>0$ so that $2\sqrt{\epsilon_0}<\epsilon_1$. Rescale the support by $x\mapsto x/(R+1)$. For sufficiently small $\de$, the full energy of the rescaled measure, and hence its regularised energy at every scale, are at most $\de^{-2\e}$. If $|\xi|\geq\de^{-\sqrt{\e}}$, take
        \[
            \rho=((R+1)^n|\xi|)^{-1}.
        \]
        Then $\de^{-2\e}\leq\rho^{-\epsilon_1}$, so the cited theorem applies at scale $\rho$ and frequency $(R+1)^n\xi$. This gives the desired decay, with $\tau$ depending only on $\kappa$ and $n$.
    \end{proof}
\end{theorem}

The same argument clearly applies when $C$ is not $\de$-separated but still a $(\de,\ka,\demm)$-set. It also applies to the uniform probability measure on a fixed constant multiple of the $\de$-neighbourhood, after an arbitrarily small increase in $\e$. Only how small $\de$ must be depends on this constant and on $R$.

\section{Proof of Theorem \ref{thm.main}}
 
\begin{steps}
    \step{Setup and selection of constants} Fix $0 < \ka < \sigma < 1$ as in the statement of Theorem \ref{thm.main}. Choose an integer $1/\kappa<n<2/\kappa$, and let $\tau>0$ be given by Theorem \ref{thm.frequencies}. Shrinking $\tau$ if necessary, choose an integer $m\geq 2$ so that $\sigma<m\tau<1$. Choose $\eta>0$ sufficiently small depending on these parameters, and then choose $\e>0$ so small that
    \[
       \e(m+1)^{ 2/\eta}\leq\eta/10.
    \]
    The constants in exponents denoted by $O(\ \cdot\ )$ may depend on $\kappa$. Multiplicative implicit constants may also depend on the fixed parameters $m,n,n_0$ below.
\step{Finding a set with small multiplicative growth}
    Let $n_0 \geq 1$ be an integer so that
\begin{equation}\label{eq.choice}
    \cn{C^{((m+1)n_0)}} \leq \dmm \cn{C^{(n_0)}}.
\end{equation}
To select $n_0$, we search along $1,m+1,(m+1)^2,\ldots$. If \eqref{eq.choice} fails at the first $\lceil 2/\eta\rceil$ choices, then
\[
    \cn{C^{((m+1)^{\lceil 2/\eta\rceil})}}
    >\de^{-\eta\lceil 2/\eta\rceil}\cn C
    \gtrsim\de^{-2}|C|.
\]
This contradicts the $O(\de^{-1})$ covering bound for a fixed bounded interval, once $\de$ is sufficiently small. We may therefore select $n_0\leq(m+1)^{\lceil 2/\eta\rceil}$, and our choice of $\e$ ensures $n_0\e\leq\eta/10$. For the remainder of the proof define $C'=C^{(n_0)}$.

    \step{Obtaining non-concentration on iterated products} Replace $C'$ with a maximal $\de$-separated subset, and then apply Lemma \ref{lem.uniform-subset} to retain a uniform subset containing at least a $\de^{\e}$-fraction of its original size. Suppose that \eqref{eq.result} fails. By the Pl\"unnecke--Ruzsa inequalities, (see \cite[Corollary 3.4]{gut}) after replacing $A$ by a subset of comparable size, we have
    $$\cn{A+B}\lesssim \demm |A|$$
    and
    $$\cn{AC'^{(n)}} \lesssim \de^{-O(nn_0)\e} |A| \leq \de^{-O(n)\eta}|A|.$$

\begin{claim}\label{claim.measures}
For each integer $1\leq k\leq m$ the set $C'^{(k)}$ is a $(\de,\ka, \dm)$-set.
\begin{proof}
    By \eqref{eq.choice}, and by extending shorter products using a fixed element of $C$, we have
    \[
        \cn{C'^{(k)}}\lesssim\dm|C'|,
        \qquad
        \cn{CC'^{(k)}}\lesssim\dm|C'|.
    \]
    Let $\de < r < 1$ and let $I$ be an interval of length $r.$ The non-concentration condition on $C$ gives at least $\gtrsim\de^\e r^{-\kappa}$ points $c_i\in C$ separated by $10r$. Since $C'^{(k)}\subset[1,\infty)$, the sets $c_i(C'^{(k)}\cap I)$ mutually disjoint. Their union is contained in $CC'^{(k)}$, and hence
    \[
        \de^\e r^{-\kappa}\cn{C'^{(k)}\cap I}
        \lesssim\cn{CC'^{(k)}}
        \lesssim\dm|C'| \leq \dm \cn{C'^{(k)}}
    \]
  The desired claim then follows. 
\end{proof}
\end{claim}
    \step{Mapping products of $C'$ into $A$ via sums}
    We record a discretised version of \cite[Lemma 2.1]{KatzTao1999}.
    \begin{lemma}
    Let $X$ be a finite set and let $A_1,\ldots,A_n\subset\R$ be finite sets. For each $1 \leq i \leq n$ let $f_i: X \rightarrow A_i$. Then
    \begin{equation}
        |\{(x_0,\ldots,x_n) \in X^{n+1}: |f_i(x_{i-1}) - f_{i}(x_i)| \leq \de \text{ for all } 1 \leq i \leq n \}| \gtrsim_n \frac{|X|^{n+1}}{\prod_{i=1}^n\cn{A_i}}.
    \end{equation}
    \begin{proof}
        Partition $\R$ into indexed intervals of length $\de$ and replace each $f_i(x)$ by the index of the interval containing it. The resulting map has at most $O(\cn{A_i})$ values. Apply \cite[Lemma 2.1]{KatzTao1999} to these maps; equality of interval indices implies that the original values differ by at most $\de$.
    \end{proof}
    \end{lemma}
    We apply this with
    $$X = A \times B \times (C')^n,$$ and
\begin{equation}
    f_{2i-1}(a,b,c_1,\ldots,c_n) = a\prod_{\ell=1}^n c_\ell,
\end{equation}
\begin{equation}
        f_{2i}(a,b,c_1,\ldots,c_n) = a+b
\end{equation}
for $1 \leq i \leq m$. This tells us that
\begin{align}
    &\left|\left\{(a_j,b_j,\tilde c_j)_{j=0}^{2m}\in\bigl(A\times B\times(C')^n\bigr)^{2m+1}:\text{\eqref{eq.prod} and  \eqref{eq.sum} hold }\right\}\right| \notag\\
    &\gtrsim \frac{|A|^{2m+1}|B|^{2m+1}|C'|^{(2m+1)n}}{\cn{AC'^{(n)}}^m\cn{A+B}^m}
    \gtrsim \de^{O(mn)\eta}|A||B|^{2m+1}|C'|^{(2m+1)n}
    \label{eq.bound}
\end{align}
where $\tilde c_j=(c_{j,1},\ldots,c_{j,n})$ and
\begin{equation}\label{eq.prod}
    \left|a_{2i-2}\prod_{\ell=1}^n c_{2i-2,\ell} - a_{2i-1}\prod_{\ell=1}^n c_{2i-1,\ell}\right| \leq \de \quad\text{for all } 1 \leq i \leq m,
\end{equation}
\begin{equation}\label{eq.sum}
    |a_{2i-1}+b_{2i-1}-a_{2i}-b_{2i}|\leq\de \quad\text{for all }1\leq i\leq m.
\end{equation}
Once $a_0$, all $b_j$, and all $c_{j,\ell}$ are fixed, these relations determine at most $O(1)^{2m}$ possibilities for $(a_1,\ldots,a_{2m})$: each successive coordinate lies in an interval of length $O(\de)$, and $A$ is $\de$-separated.

Fix $\de<\rho<1$. For $1\leq i\leq m$, the number of elements in the set above with $|b_{2i-1}-b_{2i}|<\rho$ is
$$\lesssim \demm \rho^\kappa |A| |B|^{2m+1}|C'|^{(2m+1)n}.$$
Indeed, once all other $b_j$ are fixed, the non-concentration condition on $B$ leaves at most $\lesssim\demm\rho^\kappa|B|$ choices for $b_{2i}$. We then fix $a_0$ and all $c_{j,\ell}$ and use the preceding bound on the remaining $a_j$. Choosing
$$\rho = \de^{O(mn/\kappa)\eta}$$
with a sufficiently large positive constant in the exponent, we are able to discard elements with $|b_{2i-1}-b_{2i}|<\rho$ for any $i$, and only affect \eqref{eq.bound} by a constant multiple.

Apply the pigeonhole principle to find $a_0\in A$, $b_0,\ldots,b_{2m}\in B$, and
$$\tilde c_1,\tilde c_3,\ldots,\tilde c_{2m-1},\tilde c_{2m}\in(C')^n$$
with $$|b_{2i-1}-b_{2i}|\geq\rho\quad\text{for each }1\leq i\leq m,$$
to see that
\begin{align}
&\left|\left\{(a_1,\ldots,a_{2m},\tilde c_0,\tilde c_2,\ldots,\tilde c_{2m-2})\in A^{2m}\times(C')^{mn}:\text{\eqref{eq.prod} and \eqref{eq.sum} hold}\right\}\right| \notag\\
&\gtrsim \de^{O(mn)\eta}|C'|^{mn}.
\end{align}
Let $\Omega$ be the coordinate projection of this set which removes the first $2m$ coordinates. Again, once $a_0$ and the elements of $B$ and $C'$ are fixed, there are at most $O(1)^{2m}$ choices for $(a_1,\ldots,a_{2m})$. Therefore
$$|\Omega|\gtrsim\de^{O(mn)\eta}|C'|^{mn}.$$

We now exploit \eqref{eq.prod} and \eqref{eq.sum}. It is simple algebra to show that
\[
    a_{2m}=b_{2m-1}-b_{2m}
    +\sum_{i=1}^m u_i\prod_{j=i}^m\prod_{\ell=1}^n c_{2j-2,\ell}+O(\de),
\]
where
$$u_1=\frac{a_0}{\displaystyle\prod_{j=1}^m\prod_{\ell=1}^n c_{2j-1,\ell}},$$
and
$$u_i=\frac{b_{2i-3}-b_{2i-2}}{\displaystyle\prod_{j=i}^m\prod_{\ell=1}^n c_{2j-1,\ell}},\qquad 2\leq i\leq m.$$
Write $\phi:\R^{mn}\to\R$ for the map
\[
    \phi(\tilde x_1,\ldots,\tilde x_m)
    =b_{2m-1}-b_{2m}+\sum_{i=1}^m u_i\prod_{j=i}^m\prod_{\ell=1}^n x_{j,\ell}.
\]
In particular, for each $1\leq i\leq m$ we have $\rho\lesssim|u_i|\leq 2$. Therefore we have
\begin{align}\label{eq.boundA1}
    \de^{O(mn)\eta}|C'|^{mn}&\lesssim\left|\left\{\tilde c\in(C')^{mn}:d(\phi(\tilde c),A)\lesssim\de\right\}\right|,
\end{align}
and 
\begin{align}\label{eq.boundA2}
\frac{|\{\tilde c\in(C')^{mn}:d(\phi(\tilde c),A)\lesssim\de\}|}{\max_{a\in A}|\{\tilde c\in(C')^{mn}:d(\phi(\tilde c),a)\lesssim\de\}|}\lesssim|A|.
\end{align}
    \step{Bounding the multiplicity of $\phi$ with Fourier decay}
    Fix $a\in A$. We will bound
    \begin{equation}\label{eq.solutions}
    |\{\tilde c\in(C')^{mn}:d(\phi(\tilde c),a)\lesssim\de\}|.
    \end{equation}
    We use the following kernel to detect solutions. Let $K_\de:\R\to(0,1]$ be defined by
$$K_\de(x)=\frac{1}{1+(x/\de)^2}.$$
Indeed, when $|x|\lesssim\de$, we have $K_\de(x)\sim 1$. Further, $K_\de(x)>0$. There is the well-known exact identity
$$K_\de(x)=\pi\de\int_\R e^{-2\pi\de|\xi|}e(\xi x)\,d\xi,$$
where $e(t)=e^{2\pi it}$. We may first bound \eqref{eq.solutions} by
\begin{align}
    &\lesssim\sum_{\tilde c\in(C')^{mn}}K_\de(\phi(\tilde c)-a)\notag\\
    &\lesssim\de^{-mn}\int_{(C'_\de)^{mn}}K_\de(\phi(\tilde x)-a)\,d\tilde x.
\end{align}
Here $C'_\de$ denotes the $\de$-neighbourhood of $C'$.

We perform a change of variables. For $1\leq j\leq m$ and $1\leq\ell\leq n$, define
\[
    y_{j,\ell}=\prod_{h=j}^m x_{h,\ell},\qquad
    \tilde y_j=(y_{j,1},\ldots,y_{j,n}).
\]
The Jacobian of the inverse map is
\[
    d\tilde x=\left(\prod_{j=2}^m\prod_{\ell=1}^n y_{j,\ell}^{-1}\right)d\tilde y\lesssim d\tilde y,
\]
since the $x$ coordinates lie in $[1/2,2^{2n_0}]$. Each $y_{j,\ell}$ lies in $(C'^{(m-j+1)})_{c_0\de}$ for a fixed constant $c_0=c_0(m,n_0)$. The positivity of $K_\de$ allows us to enlarge the domain and bound \eqref{eq.solutions} by
\[
    \lesssim\de^{-mn}
    \int_{\prod_{j=1}^m((C'^{(m-j+1)})_{c_0\de})^n}
    K_\de\left(b_{2m-1}-b_{2m}-a+\sum_{j=1}^m u_j\prod_{\ell=1}^n y_{j,\ell}\right)\,d\tilde y.
\]
We now insert the Fourier representation of $K_\de$. By Fubini and the triangle inequality, this is bounded by
\begin{equation}\label{eq.dominate}
    \lesssim\de^{1-mn}\int_\R e^{-2\pi\de|\xi|}
    \prod_{j=1}^m\left|
        \int_{((C'^{(m-j+1)})_{c_0\de})^n}
        e\left(\xi u_j\prod_{\ell=1}^n y_\ell\right)\,d\tilde y
    \right|\,d\xi.
\end{equation}

We first consider the high frequencies $|\xi|\geq\de^{-\eta^{1/4}}$. Fix $1\leq j\leq m$, and let $\mu_j$ be the uniform probability measure on the one-dimensional set $(C'^{(m-j+1)})_{c_0\de}$. By Claim \ref{claim.measures} we may apply Theorem \ref{thm.frequencies} with $\e=4\eta$, provided $\eta$ is small enough. Moreover, our choice of $\rho$ and the bound $|u_j|\gtrsim\rho$ ensure that
\[
    |u_j\xi|\gtrsim\rho\de^{-\eta^{1/4}}\geq\de^{-2\sqrt{\eta}}
\]
for sufficiently small $\eta$ and then sufficiently small $\de$. Thus
\begin{align}
    \left|\int_{((C'^{(m-j+1)})_{c_0\de})^n}
        e\left(\xi u_j\prod_{\ell=1}^n y_\ell\right)\,d\tilde y\right|
    &\lesssim\de^n\cn{C'^{(m-j+1)}}^n
        \left|\widehat{\mu_j^{\boxtimes n}}(-u_j\xi)\right|\notag\\
    &\lesssim\de^{n-O(n)\eta}|C'|^n|\rho\xi|^{-\tau}.
\end{align}
Multiplying these bounds and absorbing $\rho^{-m\tau}$ into the loss, we may bound the product in \eqref{eq.dominate} by
\begin{equation}\label{eq.productbound}
    \de^{mn-O(mn)\eta}|C'|^{mn}|\xi|^{-m\tau}.
\end{equation}
Since $0<m\tau<1$, the high-frequency contribution to \eqref{eq.dominate} is at most
\begin{equation}
    \lesssim\de^{1-O(mn)\eta}|C'|^{mn}
    \int_{\de^{-\eta^{1/4}}}^\infty e^{-2\pi\de\xi}\xi^{-m\tau}\,d\xi
    \lesssim\de^{m\tau-O(mn)\eta}|C'|^{mn}.
\end{equation}
For the low frequencies $|\xi|<\de^{-\eta^{1/4}}$, we trivially bound the product in \eqref{eq.dominate} by
$$\de^{mn-O(mn)\eta}|C'|^{mn}.$$
Basic calculus then bounds the low-frequency contribution by
\begin{equation}
    \lesssim\de^{-O(mn)\eta}|C'|^{mn}
    \left(1-e^{-2\pi\de^{1-\eta^{1/4}}}\right)
    \lesssim\de^{1-\eta^{1/4}-O(mn)\eta}|C'|^{mn}.
\end{equation}
Taking $\eta$ small enough that $1-\eta^{1/4}>m\tau$, both contributions are bounded by $\lesssim\de^{m\tau-O(mn)\eta}|C'|^{mn}$. Recalling \eqref{eq.boundA1} and \eqref{eq.boundA2}, we conclude that
$$|A|\gtrsim\de^{-m\tau+O(mn)\eta}>\de^{-\sg},$$
provided $\eta$ is sufficiently small and then $\de$ is sufficiently small. This contradicts $|A| \leq \de^{-\sg}$.
\end{steps}
\bibliographystyle{alpha}
	\bibliography{references}

@article {bou,
	AUTHOR = {Bourgain, J.},
	TITLE = {The discretized sum-product and projection theorems},
	JOURNAL = {J. Anal. Math.},
	FJOURNAL = {Journal d'Analyse Math\'{e}matique},
	VOLUME = {112},
	YEAR = {2010},
	PAGES = {193--236},
	ISSN = {0021-7670,1565-8538},
	MRCLASS = {11B30},
	MRNUMBER = {2763000},
	MRREVIEWER = {Serge\u{\i}\ V.\ Konyagin},
	DOI = {10.1007/s11854-010-0028-x},
}

@article {bougam,
AUTHOR = {Bourgain, Jean and Gamburd, Alex},
TITLE = {On the spectral gap for finitely-generated subgroups of {$\rm
SU(2)$}},
JOURNAL = {Invent. Math.},
FJOURNAL = {Inventiones Mathematicae},
VOLUME = {171},
YEAR = {2008},
NUMBER = {1},
PAGES = {83--121},
ISSN = {0020-9910,1432-1297},
MRCLASS = {22E30 (11B75 43A75)},
MRNUMBER = {2358056},
MRREVIEWER = {Ben\ Joseph\ Green},
DOI = {10.1007/s00222-007-0072-z},
URL = {https://0-doi-org.pugwash.lib.warwick.ac.uk/10.1007/s00222-007-0072-z},
}

@article {gut,
	AUTHOR = {Guth, Larry and Katz, Nets Hawk and Zahl, Joshua},
	TITLE = {On the discretized sum-product problem},
	JOURNAL = {Int. Math. Res. Not. IMRN},
	FJOURNAL = {International Mathematics Research Notices. IMRN},
	YEAR = {2021},
	NUMBER = {13},
	PAGES = {9769--9785},
	ISSN = {1073-7928},
	MRCLASS = {11B30 (05D10 28A99)},
	MRNUMBER = {4283564},
	MRREVIEWER = {Hanbin Zhang},
	DOI = {10.1093/imrn/rnz360},
	URL = {https://doi.org/10.1093/imrn/rnz360},
}

@article{o2025simple,
  title={Simple proofs of discretised projection theorems},
  author={O'Regan, William and Shmerkin, Pablo and Wang, Hong},
  journal={arXiv preprint arXiv:2511.21656},
  year={2025}
}

@misc{OrponenDeSaxceShmerkin2023,
  author        = {Orponen, Tuomas and de Saxc{\'e}, Nicolas
                   and Shmerkin, Pablo},
  title         = {On the {Fourier} decay of multiplicative convolutions},
  year          = {2023},
  eprint        = {2309.03068},
  archivePrefix = {arXiv},
  primaryClass  = {math.CA},
  url           = {https://arxiv.org/abs/2309.03068},
  note          = {To appear in the American Journal of Mathematics}
}

@article{KatzTao1999,
  author  = {Katz, Nets Hawk and Tao, Terence},
  title   = {Bounds on arithmetic projections, and applications
             to the {Kakeya} conjecture},
  journal = {Mathematical Research Letters},
  volume  = {6},
  number  = {5-6},
  pages   = {625--630},
  year    = {1999},
  doi     = {10.4310/MRL.1999.v6.n6.a3}
}

@article{Bourgain2005SumProduct,
  author  = {Bourgain, Jean},
  title   = {More on the sum-product phenomenon in prime fields and its applications},
  journal = {International Journal of Number Theory},
  year    = {2005},
  volume  = {1},
  number  = {1},
  pages   = {1--32},
  doi     = {10.1142/S1793042105000108}
}

@article{orpshabc,
  title   = {Projections, {F}urstenberg sets, and the {ABC} sum-product problem},
  author  = {Orponen, Tuomas and Shmerkin, Pablo},
  journal = {Journal of the American Mathematical Society},
  volume  = {39},
  number  = {3},
  pages   = {857--913},
  year    = {2026},
  doi     = {10.1090/jams/1073}
}

@article {bou03,
	AUTHOR = {Bourgain, J.},
	TITLE = {On the {E}rd\"os-{V}olkmann and {K}atz-{T}ao ring conjectures},
	JOURNAL = {Geom. Funct. Anal.},
	FJOURNAL = {Geometric and Functional Analysis},
	VOLUME = {13},
	YEAR = {2003},
	NUMBER = {2},
	PAGES = {334--365},
	ISSN = {1016-443X},
	MRCLASS = {11K55 (28A80)},
	MRNUMBER = {1982147},
	MRREVIEWER = {Ben Joseph Green},
	DOI = {10.1007/s000390300008},
	URL = {https://doi.org/10.1007/s000390300008},
}

@misc{ORegan2026,
  author        = {O'Regan, William},
  title         = {{Bourgain}'s projection theorem over normed division algebras},
journal={arXiv preprint arXiv:2603.29145},
  year          = {2026},
  eprint        = {2603.29145},
  archivePrefix = {arXiv},
  primaryClass  = {math.CA},
  doi           = {10.48550/arXiv.2603.29145},
  url           = {https://arxiv.org/abs/2603.29145}
}

@article{BourgainKatzTao2004,
  author  = {Bourgain, Jean and Katz, Nets and Tao, Terence},
  title   = {A sum-product estimate in finite fields, and applications},
  journal = {Geometric and Functional Analysis},
  volume  = {14},
  number  = {1},
  pages   = {27--57},
  year    = {2004},
  doi     = {10.1007/s00039-004-0451-1}
}

@article{Bourgain2005DiffieHellman,
  author  = {Bourgain, Jean},
  title   = {Estimates on exponential sums related to the
             {Diffie--Hellman} distributions},
  journal = {Geometric and Functional Analysis},
  volume  = {15},
  number  = {1},
  pages   = {1--34},
  year    = {2005},
  doi     = {10.1007/s00039-005-0500-4}
}

@article {kat,
	AUTHOR = {Katz, Nets Hawk and Tao, Terence},
	TITLE = {Some connections between {F}alconer's distance set conjecture
		and sets of {F}urstenburg type},
	JOURNAL = {New York J. Math.},
	FJOURNAL = {New York Journal of Mathematics},
	VOLUME = {7},
	YEAR = {2001},
	PAGES = {149--187},
	MRCLASS = {28A80 (28A75 28A78)},
	MRNUMBER = {1856956},
	MRREVIEWER = {Miguel Angel Mart\'{\i}n},
	URL = {http://nyjm.albany.edu:8000/j/2001/7_149.html},
}

@article{Shmerkin2019,
  author  = {Shmerkin, Pablo},
  title   = {On {Furstenberg}'s intersection conjecture,
             self-similar measures, and the {$L^q$} norms
             of convolutions},
  journal = {Annals of Mathematics},
  series  = {2},
  volume  = {189},
  number  = {2},
  pages   = {319--391},
  year    = {2019},
  doi     = {10.4007/annals.2019.189.2.1}
}
    \end{document}